%% file: arXiv_v2.tex
\documentclass[twoside]{article}

\usepackage[accepted]{aistats2019}
\usepackage[round]{natbib}
\renewcommand{\bibname}{References}
\renewcommand{\bibsection}{\subsubsection*{\bibname}}

\usepackage{mystyle}
\usepackage{mathrsfs} 

\usepackage[utf8]{inputenc} 
\usepackage[T1]{fontenc}    
\usepackage{hyperref}       
\usepackage{url}            
\usepackage{booktabs}       
\usepackage{amsfonts}       
\usepackage{nicefrac}       
\usepackage{microtype}      
\usepackage{float}

\hypersetup{colorlinks, citecolor=blue, filecolor=black, linkcolor=red, urlcolor=blue}

\begin{document}

%

%

\twocolumn[

\aistatstitle{Model Consistency for Learning with Mirror-Stratifiable Regularizers}



\aistatsauthor{ Jalal Fadili \And Guillaume Garrigos  \And  J\'er\^ome Malick  \And Gabriel Peyr\'e }

\aistatsaddress{ Normandie Universit\'e \And  Unversité Paris-Diderot \And CNRS and LJK, Grenoble  \And CNRS and ENS Paris} 
]

\begin{abstract}
  \input{S-abstract}
\end{abstract}

\input{S-intro}

\input{S-mirrorstrat}

\input{S-results}

\input{S-proofs}

\input{S-numerics}

\input{S-conclusion}

\bibliography{references}
\newpage
\appendix

\input{S-annex}

\end{document}

%% file: S-abstract.tex

Low-complexity non-smooth convex regularizers are routinely used to impose some structure (such as sparsity or low-rank) on the coefficients for linear predictors in supervised learning. 
Model consistency consists then in selecting the correct structure (for instance support or rank) by regularized empirical risk minimization. 
It is known that model consistency holds under appropriate non-degeneracy conditions. However such conditions typically fail for highly correlated designs and it is observed that regularization methods tend to select larger models.
In this work, we provide the theoretical underpinning of this behavior using the notion of mirror-stratifiable regularizers. 
This class of regularizers encompasses the most well-known in the literature, including the $\ell_1$ or trace norms. 
It brings into play a pair of primal-dual models, which in turn allows one to locate the structure of the solution using a specific dual certificate.
We also show how this analysis is applicable to optimal solutions of  the learning problem, and also to the iterates computed by a certain class of stochastic proximal-gradient algorithms. 

%% file: S-intro.tex

\section{Introduction}
\label{sec-intro}


\paragraph{Regularized empirical risk minimization.}

We consider a general set-up for supervised learning where, given an input/output space $\Xx \times \Yy$ endowed with a probability measure $\rho$, one wants to learn an estimator $f : \Xx \rightarrow \Yy$ satisfying $f(x) \approx y$ for $\rho$-a.e. pair of data $(x,y) \in \Xx \times \Yy$.
We restrict ourselves to the case where $\Xx \times \Yy = \RR^p \times \RR$, with $p$ being the dimension of the feature space, and we search for an estimator that is linear in $\Xx$, meaning that $f$ can be written $f_{w_0}(x)= \dotp{x}{\weight_0}$ for some coefficient vector $\weight_0 \in \RR^p$.
A standard \textit{modeling assumption} is that, among the minimizers of a quadratic expected risk, $w_0$ possesses some form of simplicity or low-complexity (e.g. sparsity or low-rank).
In other words, $w_0$ is assumed to be the unique solution~of
\begin{equation}\label{Eq:primal expected problem}
\min_{w \in \RR^p} \!\enscond{\!R(w)\!\!}{\!\!w \in \uArgmin{w' \in \RR^p} \EE_\rho\!\left[(\dotp{w'}{\xp} - \yp)^2\right]\!} \tag{$\text{P}_0$}
\end{equation}
where $R : \RR^p\rightarrow \RR \cup \{+\infty\}$ is a proper lower semi-continuous (l.s.c.) convex regularizer, and $\EE_\rho\left[{\cdot}\right]$ is the expectation of the random variable $(\xp,\yp)$ w.r.t. the probability measure $\rho$.


%
In practice \eqref{Eq:primal expected problem} cannot be solved directly because one does not have access to $\rho$;
 only a sequence of $n$ independent and identically distributed (i.i.d.) pairs $(x_i,y_i)_{i=1}^n$ sampled from $\rho$ is available.
The conventional approach is then to consider a solution $\widehat \weight_{\lambda,n}$ of a penalized empirical risk minimization (ERM) of the form
\eql{\label{eq-empirical-risk-bis}
\min_{\weight \in \RR^p} \lambda R(\weight) + \!\frac{1}{2 n}\sum_{i=1}^n \left(\dotp{x_i}{\weight}-y_i \right)^2\!.
	\tag{$\text{P}_{\lambda,n}$} 
}
The regularization parameter  $\la>0$ is tuned as a (decreasing) function of $n$, balancing appropriately between fitting the data and inducing some desirable property promoted by the regularizer $R$.




\paragraph{Tracking the structure of the solution.}

A theoretical question in statistical learning is to understand how close a solution $\widehat \weight_{\lambda,n}$ of~\eqref{eq-empirical-risk-bis} comes to $w_0$. 
If $\widehat \weight_{\lambda_n,n} \rightarrow \weight_0$ (convergence being usually considered in probability) as $n \rightarrow +\infty$ with $\lambda_n \rightarrow 0$, then the estimator is said to be \emph{consistent}. One is also generally interested in stating estimation rates, and a linear estimation rate corresponds to $\norm{\widehat \weight_{\lambda_n,n} - \weight_0} \sim n^{-\frac{1}{2}}$ (to be understood in probability). 
Note that we are here discussing guarantees on the estimation risk and not on the prediction risk (i.e. on $\weight$ and not on $f_w(x)=\dotp{w}{x}$), which is more challenging. 
%
In this paper, we investigate \emph{model consistency}, that is, whether $\weight_{\lambda_n,n}$ and $w_0$ share the same structure for appropriately chosen $\lambda_n$ and $n$ large enough. 
Existing results on the subject heavily rely on a non-degeneracy condition at $w_0$, which is often referred as an ``irrepresentable condition'' (see more details and references in Remark~\ref{rem-irrepresentable-cond}). In this case, one can show that for $n$ large enough and $\la_n \sim n^{-1/2}$, model consistency holds; see for $\ell_1$~\citep{zhao2006model}, $\ell_1$-$\ell_2$~\citep{bach2008group}, nuclear norm~\citep{bach2008trace} and more generally for the class of partly-smooth functions~\citep{2014-vaiter-ps-consistency}. 
The first goal of this paper is to go one step further by formally analyzing the general and challenging case where the non-degeneracy assumption cannot be guaranteed.

\paragraph{Tracking the structure of proximal algorithms.}
Similar consistency questions arise for the approximations of solutions computed by stochastic proximal algorithms used to solve \eqref{eq-empirical-risk-bis}.
Many non-smooth low-complexity structure-promoting regularizers are such that their proximal operator is easy to compute either explicitly (as for the $\ell_1$ norm or the trace norm) or  approximately to good precision (as for the total variation in one-dimension). Proximal-gradient algorithms are then the methods of choice for solving the structured optimization problem \eqref{eq-empirical-risk-bis}. For large-scale machine learning problems, 
one would typically prefer stochastic versions of these algorithms, which need only one observation to proceed with the iterate; see e.g. \citep{DefBacLac14,XiaZha14}. The second goal is then to understand 
if these iterates and $w_0$ share the same structure induced by $R$. This complements the existing convergence analysis of these algorithms; pointers to relevant literature are given in Section~\ref{sec-main-results}.

\paragraph{Paper organization.}
As explained above, this paper has two goals about general model consistency for (i) regularized learning models and (ii) stochastic algorithms for solving  them. The low-complexity induced by popular regularizers reveals primal-dual partitions which  allow us to localize optimal solutions and track iterates. Section~\ref{S:mirror-strat} recalls the notion of mirror-stratifiable regularizers which provides this structural complexity partition. 
Then Section~\ref{sec-main-results} states our model recovery results and discusses their originality with respect to the existing literature. The rationale and the milestones of the proofs are sketched in Section~\ref{S:sketch proof}; details and technical results are established in the supplementary material. Finally Section~\ref{sec-numerics} provides numerical illustrations of our results, giving theoretical justification of typical observed behaviors of stochastic algorithms.

%% file: S-mirrorstrat.tex
\section{Low-complexity models}\label{S:mirror-strat}

\paragraph{Low-complexity and stratification.}

In this paper, we study model consistency for a large class of regularizers, and under few structural assumptions. 
Our results strongly rely on duality arguments, and on a structure induced by $\partial R$ (where $\partial R$ is the subdifferential of $R$).
To track the structure of solutions,  we introduce an appropriate {\emph{stratification}}  
 $\Strat = \{M_i\}_{i\in I}$ of $\dom (\partial R) \subset \RR^p$ (where $\dom(\partial R) \eqdef \enscond{w \in \RR^p}{\partial R(w) \neq \emptyset}$), which
 is a finite partition such that for any strata  $\Man$ and $\Man'$ 
 \begin{equation*}
 ~\Man \cap \cl(\Man') \neq \emptyset ~\Rightarrow \Man \subset \cl(\Man')
\end{equation*}
(where $\cl $ stands for the topological closure of the set).
Because this is a partition, any element $w \in \dom(\partial R)$ belongs to a unique stratum, which we denote $M_w$.
A stratification also induces a partial ordering $\leq$ as follows
\eql{\label{eq-order-strata}
	\Man\!\leq\!\Man' 
	\Longleftrightarrow
	\Man\!\subset\!\cl(\Man')
	\Longleftrightarrow
	\Man\!\cap\!\cl(\Man')\!\neq\!\emptyset .
}
With such ordering, it is natural to see some strata as being ``smaller'' than others, and, by extension, to say that the elements of such small strata have a low-complexity.

\begin{example}\label{Ex:strata}
Most regularizers $R$ used in machine learning naturally come up with  a stratification, in the sense that they promote solutions belonging to small (for the relation $\leq$) strata $M$.
\begin{itemize}
	\item Lasso~\citep{tibshirani1996regre}: the simplest example is the $\ell_1$ norm where $R(w)=\sum_i |w_i|$, where the strata are the sets of vectors $M_I=\enscond{w \in \RR^p}{\supp(w)=I}$, where $I \subset \ens{1,\cdots,p}$.
	\item Nuclear (a.k.a. trace) norm~\citep{fazel2002matrix}: this is another popular example where $R(w)$ is the $\ell_1$ norm of the singular values of $w$, and where the strata are the manifolds of fixed-rank matrices: $M_r=\enscond{w \in \RR^{p_1 \times p_2}}{\rank(w)=r}$, where $r \in \ens{0,\cdots,\min(p_1,p_2)}$.
	\item Many other examples fall within this class of regularizers. For instance the $\ell_1$-$\ell_2$-norm to promote group-sparsity~\citep{yuan2005model}, or the fused Lasso~\citep{tibshirani2005sparsity}. 
Yet another example is the total variation semi-norm $R(w)=\norm{D w}_1$ where $D$ is a discrete approximation to the ``gradient'' operator (on a regular grid or on a graph); in this case, the strata are defined by piecewise constant vectors sharing the same jump set (edges in signals or images).
\end{itemize}
\end{example}

\paragraph{Mirror-Stratifiable Regularizers.}

All the classical regularizers mentioned in Example~\ref{Ex:strata} have moreover a strong relation between their primal and dual stratifications. 
These primal-dual relations are defined through the following correspondence operator $\Jj_R$ between  subsets $S \subset \RR^p$,  
\eq{
	\Jj_R(S) \eqdef \bigcup_{x \in S} \ri(\partial R(x)),
}
where $\ri$ denotes the relative interior of a convex set.
Following~\cite{FadMalPey17}, we define mirror-stratifiabilty as follows.
\begin{defn}\label{D:mirror-strat}
Let $R$ be a proper lsc and convex function and $R^*$ its Legendre-Fenchel conjugate. $R$ is mirror-stratifiable with respect to a (primal) stratification $\Strat=\{\Man_i\}_{i\in I}$ of $\dom(\partial R)$ and a (dual) stratification $\Strat^*=\{\Man^*_i\}_{i\in I}$ of $\dom(\partial R^*)$ if $\Jj_R: \Strat \to \Strat^*$ is invertible with inverse $\Jj_{R^*}$ and $\Jj_R$ is decreasing for the relation $\leq$ defined by \eqref{eq-order-strata}.
\end{defn}
This structure finds its roots in~\citep{daniilidis2013orthogonal}, which introduces the tools to show that polyhedral functions, as well as spectral lifting of polyhedral functions, are mirror-stratifiable. In particular, all popular regularizers mentioned above ($\ell_1$ norm, $\ell_1$-$\ell_2$ mixed norms, nuclear norm, total variation semi-norm) are mirror-stratifiable; see~\citep{FadMalPey17}.

\begin{example}\label{Ex:mirror}
Let us illustrate this notion in the case $R=\Vert \cdot \Vert_1$.
As mentioned in Example~\ref{Ex:strata}, the strata $M_I$ of $\dom(\partial R) = \RR^p$ are sets of sparse vectors, with prescribed support.
In the dual, $\dom(\partial R)^*$ is the unit $\ell_\infty$-ball, which can be naturally stratified by sets of vectors in $[-1,1]^p$ with a prescribed active set.
More precisely, if we define 
\begin{equation*}
\Active(\eta) \eqdef \enscond{i \in \ens{1,\cdots,p}}{ \vert \eta_i \vert = 1},
\end{equation*}
 then these strata are of the form $M^*_I=\enscond{\eta \in [-1,1]^p}{\Active(\eta) = I}$.
It is then an easy exercise to verify that the following correspondence operators $\Jj_R$ and $\Jj_{R^*}$ induce a decreasing bijection between the dual strata $M^*_I$ and the primal strata $M_I$, meaning that:
\begin{align*}
(\forall I,J \subset \ens{1,\cdots,p}) \quad &\Jj_R(M_I) = M^*_I, \  \Jj_{R^*}(M^*_I)=M_I \\
\text{ and } \ &I\subset J \Leftrightarrow M_I \leq M_J \Leftrightarrow M_I^* \geq M_J^*.
\end{align*}
All the regularizers in Example~\ref{Ex:strata} work in the same way.
For instance, for the nuclear norm, the strata~$M_r$ made of rank-$r$ matrices are in correspondence with strata $M^*_r$ made of matrices having exactly $r$ singular values equal to $1$, and the others being of smaller amplitude.
\end{example}

%% file: S-results.tex
\section{Main results}\label{sec-main-results}

We study model consistency by bypassing unrealistic assumptions (e.g., irrepresentable-type condition)
and thus obtain flexible theoretical results.
Throughout this paper, we only assume the following hypotheses:
\begin{equation}\label{H:hypothesis on model}
\tag{$\text{H}_\text{M}$}\begin{cases}
\text{$R$ is mirror-stratifiable}, \\
\text{$R$ is bounded from below}, \\
\text{$w_0$ is the unique solution of~\eqref{Eq:primal expected problem}}.\\
\end{cases}
\end{equation}
Under~\eqref{H:hypothesis on model}, we establish general model consistency results of optimal solutions of the regularized ERM problem~\eqref{eq-empirical-risk-bis} (in Section\;\ref{sec-consistency-1}), and of iterates of stochastic proximal algorithms to solve it (in Section\;\ref{sec-consistency-2}). 
We also discuss how these results encompass the existing model consistency results (in Section\;\ref{sec-relation}).

Our analysis leverages the strong primal-dual structure of mirror-stratifiable regularizers, which is our key tool to localize the active strata at the solution of~\eqref{eq-empirical-risk-bis}, even in the case where the irrepresentable condition is violated. 
We  show  that an enlarged model consistency holds, where the identified structure lies between the ideal one (the structure of $w_0$) and a worst-case one controlled by 
a particular dual element (the so-called dual vector/certificate) 
\[
\eta_0\in\partial R(w_0),
\] 
defined as the optimal solution\footnote{Though we do not assume $C$ to be invertible, $\eta_0$ is indeed unique since $\Ker C^\dagger = \Im C^\perp$. In the case where  $C$ is invertible, $\eta_0$ coincides with the element of $\partial R(w_0)$ having minimal norm, in the metric induced by\;$C^{-1}$.}
\begin{equation}\label{E:dual expected pb}
\eta_0=\Argmin\enscond{\dotp{C^\dagger \eta}{\eta}\!\!}{\!\!\eta  \in \partial R(w_0) \cap \Im C}
\tag{$\text{D}_\text{0}$}
\end{equation} 
where 
$C \eqdef \EE_\rho \left[\xp \xp^\top\right] \in \RR^{p \times p}$ is the expected (non-centered) covariance matrix, and $C^\dagger$ denotes its Moore-Penrose pseudo-inverse. The role of $\eta_0$ in sensitivity analysis of regularized ERM problems is well-known, but has been always done under a non-degeneracy assumption (see forthcoming discussions in Remark\;\ref{rem-irrepresentable-cond} and Section~\ref{sec-relation}).


\subsection{Model consistency for regularized ERM}\label{sec-consistency-1}
Our first contribution, Theorem~\ref{T:identification problem proba} below, states that for an appropriate regime of $(\lambda_n,n)$, one can precisely localize with probability $1$ the active stratum at $\widehat\weight_{\lambda_n,n}$ between a minimal active set associated to $\weight_0$ and a maximal one controlled by the dual vector $\eta_0$. 
In the special case of $\ell_1$ minimization, this means that, almost surely, the support of $\widehat\weight_{\lambda_n,n}$ can be larger than that of $\weight_0$ but cannot be larger than the extended support characterized by $\Active(\eta_0)$. This holds provided that $\lambda_n$ decreases to $0$ with\;$n$, but not too fast to account for errors stemming from the finite sampling.

\begin{thm}
\label{T:identification problem proba}
Assume that \eqref{H:hypothesis on model} holds, and suppose that $\EE_\rho\left[\Vert \xp \Vert^4\right] <+\infty$ and $\EE_\rho\left[\vert \yp \vert^4\right] < + \infty$.
Let $(\lambda_n)_{n\in\NN} \subset ]0,+\infty[$ be such that 
\[
\lambda_n \to 0 \text{\quad with\quad $\lambda_n \sqrt{n/(\log \log n)} \to +\infty$}.
\]
Then, for $n$ large enough, the following holds with probability $1$: 
\vspace*{-2ex}
\begin{equation}\label{Eq:sandwich strata}
 M_{\weight_0} \leq M_{\widehat \weight_{\lambda_n,n}} \leq \Jj_{R^*}(M_{\eta_0}^*).
\end{equation}
\vspace*{-1.5ex}
\end{thm}

\begin{example}\label{Ex:mirror-strat dual certificate}
Using the notations of Example \ref{Ex:strata} and \ref{Ex:mirror}, the enlarged consistency \eqref{Eq:sandwich strata} specializes to
\begin{align*}
&\supp(w_0) \subset \supp(\widehat \weight_{\lambda_n,n}) \subset \Active(\eta_0),\\
&\rank(w_0) \leq \rank(\widehat \weight_{\lambda_n,n}) \leq \#\enscond{s \in \sigma(\eta_0)}{\vert s \vert=1},
\end{align*}
for the $\ell_1$ norm and the nuclear norm, respectively. $\sigma(\eta_0)$ denotes the vector of singular values of $\eta_0$.
\end{example}

The theorem guarantees that we have an enlarged model consistency, as soon as enough data is sampled.
The first interest of this result is the finite identification, compared to the existing asymptotic results (even if the level of generality does not allow us to provide a bound on $n$); we discuss this in Section\;\ref{sec-relation}. The second and main advantage of our result is that it does not require any unrealistic non-degeneracy assumption. We explain this point in the next two remarks, by looking at the usual assumption and how it often fails to hold in high dimension.

\begin{remark}[Irrepresentable condition and exact model consistency]\label{rem-irrepresentable-cond}
	If it is furthermore assumed that 
	\begin{equation}\label{D:irrepresentable condition}
	\eta_0 \in \ri(\partial R(w_0)), \tag{IC}
	\end{equation}
	then it follows from Definition \ref{D:mirror-strat} that $M_{\weight_0}=\Jj_{R^*}(M_{\eta_0}^*)$.
	In that setting, the consistency \eqref{Eq:sandwich strata} just gives exact model consistency
	\[
	M_{\weight_0}=M_{\widehat \weight_{\lambda_n,n}}.
	\] 
	This relative interiority assumption \eqref{D:irrepresentable condition} corresponds exactly to the ``irrepresentable condition'' which is classical in the learning literature \citep{zhao2006model},\citep{bach2008group},\citep{bach2008trace}.
	Without this non-degeneracy hypothesis, we cannot expect to have  exact model consistency (this is for instance illustrated in Section~\ref{sec-numerics}). The above theorem shows that there is still an approximate optimal model consistency, with two extreme strata  fully characterized by the primal-dual pair $(\weight_0,\eta_0)$. 
	Our result is thus able to explain what is going on in the intricate situation where  \eqref{D:irrepresentable condition} is violated.
	%
\end{remark}

\begin{remark}[When the irrepresentable condition fails]\label{R:compressed sensing}
The originality and interest of our model consistency result is that  condition \eqref{D:irrepresentable condition} is not required to hold, since it is usually not valid in the context of large-scale learning.
Let us give some insights on this condition in the specific case of $\ell_1$-regularized problems.
For instance, if the $x_i$'s are drawn from a standard Gaussian i.i.d. distribution, the compressed sensing literature provides sample thresholds depending on the dimension $p$ and the sparsity level $s = \Vert w_0\Vert_0$. 
In this scenario, it is known that uniqueness in~\eqref{H:hypothesis on model} holds for $n > 2s\log(p/s)$ \citep{AmeLotCoyTro14}, while the irrepresentable condition holds only for $n > 2s\log(p)$ \citep{CanRec13}: the gap between these thresholds corresponds to the case where \eqref{D:irrepresentable condition} fails.
Observe nevertheless that these results rely on the assumption that the features are incoherent (here Gaussian i.i.d.), which is not likely to be verified  in a  learning scenario, where they are typically highly correlated.
A setting with a coherent operator $C$ is that of deconvolution, where $C$ is a (discerete) convolution operator associated to a smooth kernel, which is widely studied in the signal/image processing literature (in particular for the super-resolution). 
In this case, one can exactly determine the largest manifold {$\Jj_{R^*}(M^*)$} involved in \eqref{Eq:sandwich strata}, see \citep{DuvPey17}.
\end{remark}

\subsection{Model consistency for stochastic proximal-gradient algorithms}\label{sec-consistency-2}

Our second main result describes model consistency for the iterates generated by a stochastic algorithm. 
In our situation, the general (relaxed) stochastic proximal gradient algorithm for solving \eqref{eq-empirical-risk-bis}  
reads, starting from any initialization $\widehat\weight^0$, at iteration $k$:
\begin{equation}\label{algorithm}
\begin{cases}
\widehat d^{k} = (\langle \widehat w^k, x_{i(k)} \rangle -y_{i(k)}) x_{i(k)} + \widehat \varepsilon^k, \\ 
\widehat z^{k}   =  \Prox_{\gamma_k \la R}(
		\widehat w^k - \gamma_k \widehat d^k ), 
		\\
		\widehat w^{k+1}  = (1-\alpha_k) \widehat w^k + \alpha_k \widehat z^{k},
\end{cases}
\tag{$\text{RSPG}$}
\end{equation}
where $(x_{i(k)},y_{i(k)})$ are independent random variables drawn among $(x_i,y_i)_{i=1}^n$, $\gamma_k \in ]0,+\infty[$ and $\alpha_k \in ]0,1]$ are respectively deterministic stepsize and relaxation parameters. 
As it is, the iteration is written in an abstract way, since we do not specify how to define the random $\RR^p$-valued variables $\widehat \varepsilon^k$.
But as we explain below, several known stochastic methods can be written under the form of \eqref{algorithm} when $\alpha_k \equiv 1$.

\begin{example}\label{Ex:stochastic methods choice epsilon}
If one takes $\widehat{\varepsilon}^k \equiv 0$, then \eqref{algorithm} becomes simply the proximal stochastic gradient method (Prox-SGD).
Variance-reduced methods, like the SAGA algorithm \citep{DefBacLac14}, or the Prox-SVRG algorithm (with option I) \citep{XiaZha14}, also fall  into this scheme. 
For these algorithms the idea is to take $\widehat{\varepsilon}^k$ as a  combination of previously computed estimates of the gradient, in order to reduce the variance of $\widehat d^{k}$.
For instance, SAGA corresponds to the choice:
	\begin{align*}
	&\widehat{\varepsilon}^k = \frac{1}{n} \sum\limits_{i=1}^n g_{k,i} - g_{k,i(k)}
	\end{align*}	
	\text{where the stored gradients are updated as} 
		\begin{align*}
	&g_{i,k} \eqdef \begin{cases}
	(\langle \widehat w^k, x_{i(k)} \rangle -y_{i(k)}) x_{i(k)} & \text{if } i = i(k) \\
	g_{k-1,i} & \text{else.}
	\end{cases}
	\end{align*}
\end{example}

We show in Theorem \ref{thm-ident-sgd} that for $n$ large enough and $\lambda_n$ appropiately chosen, we can identify after a finite number of iterations of \eqref{algorithm} an active stratum, which is again localized between two strata controlled by $w_0$ and $\eta_0$, respectively.
For this result to hold, we have to make some reasonable assumptions on algorithm \eqref{algorithm}.
We need first to make hypotheses on the parameters $\alpha_k$, $\gamma_k$, $\widehat{\varepsilon}^k$, to ensure that the iterates of \eqref{algorithm} converge to a solution of \eqref{eq-empirical-risk-bis}. Such hypotheses have been investigated in \citep{CombettesStochastic16,RosVilVu16,AtcForMou17} to establish useful convergence results. Beyond convergence, we study 
\textit{structure identification} of these algorithms. It is known that convergence is not enough for model consistency of iterates.
For instance, the classical proximal stochastic gradient method (corresponding to the case $\widehat{\epsilon}^k \equiv 0$) is known to fail at generating sparse iterates for the case $R = \Vert \cdot \Vert_1$; see below Example \ref{Ex:algos} for discussions and references.
To ensure the identification of low-dimensional strata, we  require some control on the variance of the descent direction, by acting either on the parameters $\gamma_k$ and $\alpha_k$, or by wisely controlling  $\widehat{\epsilon}^k$.
Before stating formally this set of hypotheses, we introduce $L_n\eqdef (1/n)\Vert \sum_{i=1}^n  x_i x_i^* \Vert$ and 
the $\sigma$-algebra $\Ff_k \eqdef \sigma(\widehat{w}^1,\dots , \widehat{w}^k)$ generated by the first $k$ iterates.
\begin{equation}\label{H:hypothesis on algorithm}
\begin{cases}
\sigma_k \in [0,+\infty[, \ \alpha_k \in ]0,1], \ \gamma_k \in ]0,2/L_n[  \\[0.8ex]
\EE\left[\widehat{\varepsilon}^k | \mathcal{F}_k\right] = 0, \ \Var\left[ \widehat{d}^k | \Ff_k \right] \leq \sigma_k^2 \\
\ \widehat{d}^k - \EE\left[\widehat{d}^k | \mathcal{F}_k\right] \text{ converges a.s. to } 0 \\[0.2ex]
 \sum_{k=1}^{\infty} \alpha_k\gamma_k^2 \sigma_k^2 < + \infty \\[1ex]
 \norm{\widehat w^{k+1}- \widehat w^k} = o(\alpha_k \gamma_k) \ \text{a.s.}
\end{cases}
\tag{$\text{H}_\text{A}$}
\end{equation}

Let us briefly discuss these hypotheses.
The second line in~\eqref{H:hypothesis on algorithm} imposes some control on the variance of~$\widehat{d}^k$.
The fourth line asks for a fine balance between the parameters $\alpha_k$, $\gamma_k$ and $\sigma_k$.
For instance, one could take $\alpha_k$ and $\gamma_k$ to be constant, and work essentially on $\widehat \varepsilon^k$ to ensure that $\sigma_k \in \ell^2(\NN)$ and $\widehat{w}^k$ is a.s. asymptotically regular.
Instead, one could consider an algorithm where $\sigma_k$ does not vanish, but with appropriately decreasing step-sizes: $\gamma_k$ and $\alpha_k$ should be carefully chosen to guarantee that the fourth row of~\eqref{H:hypothesis on algorithm} holds.


\begin{thm}\label{thm-ident-sgd}
Assume that \eqref{H:hypothesis on model} holds, and suppose that $\EE\left[\Vert \xp \Vert^4\right] <+\infty$ and $\EE\left[\vert \yp \vert^4\right] < + \infty$. Let $(\lambda_n)_{n\in\NN} \subset ]0,+\infty[$ be such that 
\[
\lambda_n \to 0 \quad\text{with $\lambda_n \sqrt{n/(\log \log n)} \to +\infty$}.
\]
Then, for $n$ large enough, if $(\widehat w^k)_{k\in\NN}$ is generated by~\eqref{algorithm} under assumption  \eqref{H:hypothesis on algorithm}, then for $k$ large enough: 
\vspace*{-1ex}
	\begin{equation*}
		M_{\weight_0} \leq M_{\widehat z^k} \leq \Jj_{R^*}(M_{\eta_0}^*) \quad \text{ holds almost surely}.	
	\end{equation*}
	\vspace*{-1.5ex}
\end{thm}

\begin{example}\label{Ex:algos}
Let us look at two instances of \eqref{algorithm}.
\begin{itemize}
	\item The SAGA (resp. Prox-SVRG) algorithm is shown to verify \eqref{H:hypothesis on algorithm} in \citep{PooLiaSch18}, provided that $\alpha_k \equiv 1$ and $\gamma_k \equiv \gamma = 1/(3L_n)$ (resp. $\gamma_k \equiv \gamma$ taken small enough).
	\item The proximal stochastic gradient method (Prox-SGD) is a specialization of \eqref{algorithm} with $\alpha_k \equiv 1$,  $(\gamma_k)_{k \in \NN} \in \ell^2(\NN) \setminus \ell^1(\NN)$ and $\widehat{\varepsilon}^k \equiv 0$.
	If the iterates are bounded, the second line of \eqref{H:hypothesis on algorithm} automatically holds by the (strong) law of large numbers.
	Nevertheless, this algorithm does not satisfy the conclusions of Theorem \ref{thm-ident-sgd}: this was observed in \citep{Xia10,LeeWri12,PooLiaSch18}, and is illustrated in Section \ref{sec-numerics}. 
	A simple explanation is that for this algorithm, $\sigma_k$ does not converge to 0, which is why we need to impose that the stepsize $\gamma_k$ tends to zero.
	Even if it can be shown that $\norm{\widehat w^{k+1}- \widehat w^k} / \gamma_k$ is bounded, it cannot be ensured that it is $o(1)$, which breaks the last hypothesis in \eqref{H:hypothesis on algorithm}. Thus Theorem~\ref{thm-ident-sgd} does not apply in agreement with the observed behaviour of the SGD algorithm.
\end{itemize}
\end{example}

\subsection{Relation to previous results.}\label{sec-relation}
Model consistency of the regularized ERM has already been investigated for special cases ($\ell_1$~\citep{zhao2006model}, $\ell_1$-$\ell_2$~\citep{bach2008group}, or nuclear norm~\citep{bach2008trace}) and for 
the class of partly-smooth functions \citep{2014-vaiter-ps-consistency}.
The existing results hold asymptotically in probability, e.g.\;of the form
\begin{equation}\label{Eq:asymptotic consistency}
\lim\limits_{n \to +\infty} \PP \left( M_{\weight_0}=M_{\widehat \weight_{\lambda_n,n}} \right) = 1,
\end{equation} 
while we show that the consistency \eqref{Eq:sandwich strata} holds almost surely, as soon as enough data is sampled.
Nevertheless, our result lacks a quantitative estimation of how large $n$ should be for the identification to hold.
As a comparison, \citep{2014-vaiter-ps-consistency} shows that the probability in \eqref{Eq:asymptotic consistency} converges as $1-n^{-1/2}$,  but the result heavily relies on the assumption that \eqref{D:irrepresentable condition} holds (which prevents the solution from ``jumping'' between the strata $M_{w_0}$ and $\Jj_{R^*}(M^*_{\eta_0})$). Such qualitative estimates cannot be derived in our more general results without stronger assumptions and/or structure, which we want to avoid.

Compared to previous works, a chief advantage of our model consistency results is thus to avoid making an assumption which often fails to hold in high dimension. Indeed, as explained in Remark \ref{R:compressed sensing} and Section~\ref{sec-numerics}, the above-mentioned existing results hold under the irrepresentable condition \eqref{D:irrepresentable condition}; and many of these also assume that the expected covariance matrix $C = \EE_\rho\!\left[\xp \xp^\top\right]$ is invertible. The first work to deal with model consistency for a large class of functions without the irrepresentable condition assumption is \citep{FadMalPey17}, which introduces of the notion of mirror-stratifiable functions, 
from which the authors derive identification properties of a \textit{deterministic} penalized problem.
Our Theorem \ref{T:identification problem proba} comes with a similar flavor, but extended to a supervised learning scenario and random sampling, which brought technical challenges as detailed in Section\;\ref{S:sketch proof}.



Finite activity identification for  {stochastic} algorithms has been a topic of interest in the past years. \citep{Xia10} made the observation that Prox-SGD has not the identification property for the $\ell_1$ case.
Instead, finite activity identification was proved by \cite{LeeWri12}, for the regularized dual averaging (RDA) method, and by \cite{PooLiaSch18}, for the SAGA and Prox-SVRG algorithms.
For these two papers, the regularizer $R$ is assumed to be partly-smooth, and a non-degeneracy assumption is made.
Again, Theorem \ref{thm-ident-sgd} does not need such an assumption. We also propose a general set of hypotheses \eqref{H:hypothesis on algorithm} encompassing all these algorithms and beyond: this allows an explanation for why Prox-SGD fails (see Example \ref{Ex:algos}), and could be used to analyze other algorithms than SAGA or Prox-SVRG.

%% file: S-proofs.tex
\section{Sketch of  proofs}\label{S:sketch proof}

Our model consistency results follow from a sequence of results controlling the behaviour of optimal solutions and of iterates of algorithms. In this section, we sketch the rationale and the milestones of the proof; the proof of the two intermediate technical results are given in the supplementary material.

The core of the proofs rely on~\citep[Theorem~1]{FadMalPey17} about sensivity analysis of mirror-stratifiable functions. 
We state this result here in a modified form that is adapted to our analysis.
\begin{prop}\label{T:mirror strat identification}
Let $R$ be mirror-stratifiable.
Then, there exists $ \delta > 0$ such that for any pair $\eta\in\partial R(w)$,
\[
\max\{ \norm{ w - w_0} , \norm{\eta - \eta_0} \} \leq \delta \Rightarrow M_{w_0} \leq M_{ w} \leq \Jj_{R^*}(M_{\eta_0}^*).
\]
\end{prop}
\begin{proof}
 Suppose for contradiction that no such $\delta$ exists. Let $(\delta_k)_{k \in \NN} \subset ]0,+\infty[$ and $(w^k,\eta^k)_{k \in \NN} \subset \Gr(\partial R)$ be such that $\delta_k \downarrow 0$, $\max(\norm{w^k-w_0},\norm{\eta^k-\eta_0}) \leq \delta_k$, but where $M_{w^k}$ does not satisfy the claimed inequalities. Then $(w^k,\eta^k) \to (w_0,\eta_0)$ as $k \to \infty$, and $(w_0,\eta_0) \in \Gr(\partial R)$ by definition in \eqref{E:dual expected pb}.
Upon applying \citep[Theorem~1]{FadMalPey17}, we have that $M_{w_0} \leq M_{w^k} \leq \Jj_{R^*}(M_{\eta_0}^*)$ for $k$ sufficiently large. This is a contradiction with the choice of $w^k$.
\end{proof}

Concerning Theorem \ref{T:identification problem proba}, we  introduce the notations
\begin{align*}
\widehat C_n &\eqdef \frac{1}{n}\sum_i x_ix_i^\top \in \RR^{p \times p},\\
\widehat u_n &\eqdef \frac{1}{n}\sum_i y_i x_i \in \RR^p, \quad 
u \eqdef \EE_\rho\left[\yp \xp\right],
\end{align*} 
which allows us to rewrite problems \eqref{Eq:primal expected problem} and
\eqref{eq-empirical-risk-bis} in a compact form: 
\begin{align*}
\ens{w_0} &= \uArgmin{w \in \RR^p,Cw=u}R(w),\\
\widehat w_{\lambda,n} &\in \uArgmin{w \in \RR^p} \lambda R(w) + \frac{1}{2}\langle \widehat C_n w,w \rangle - \langle  \widehat u_n,w \rangle.  
\end{align*}
The optimality conditions for \eqref{eq-empirical-risk-bis} allow to derive:
\begin{equation}\label{ident1}
\widehat \eta_{\lambda_n,n}  \in \partial R(\widehat w_{\lambda_n,n}), \quad \widehat \eta_{\lambda_n,n} \eqdef \frac{\widehat u_n - \widehat C_n \widehat w_{\lambda_n,n}}{\lambda_n}.
\end{equation}

In view of Proposition~\ref{T:mirror strat identification}, establishing Theorem\;\ref{T:identification problem proba} essentially boils down to showing the following proposition. The proof of this proposition requires technical lemmas to control the interlaced effects of convergence and sampling; see the supplementary material for details.

\begin{prop}\label{P:cv problem}
Under the assumptions of Theorem\;\ref{T:identification problem proba}, $(\widehat w_{\lambda_n,n},\widehat \eta_{\lambda_n,n}) {\underset{n\to +\infty}{\longrightarrow}}$  $(w_0,\eta_0)$ almost surely.
\end{prop}

From  Proposition \ref{P:cv problem}, we deduce that there exists $N \in \NN$ such that for all $n \geq N$:
\begin{equation}\label{tis1.5}
\max\{ \norm{\widehat w_{\lambda_n,n} - w_0} , \norm{\widehat \eta_{\lambda_n,n} - \eta_0} \} \leq \delta/2 \quad a.s.
\end{equation}
Using Proposition \ref{T:mirror strat identification}, we deduce that \eqref{Eq:sandwich strata} holds a.s. for all $n \geq N$.
To prove Theorem \ref{thm-ident-sgd}, we keep $n \geq N$  fixed, and consider $(\widehat w^k)_{k \in \NN}$ to be generated by the \eqref{algorithm} algorithm.
Using the definition of $\widehat w^{k+1}$, we can write
\begin{align}
\widehat z^{k} &= \widehat w^k + \frac{\widehat w^{k+1} - \widehat w^k}{\alpha_k}, \label{tis2.0} \\ 
\widehat w^k - \gamma_k\widehat{d}^k &\in \widehat z^k + \gamma_k \lambda_n \partial R(\widehat z^k).\label{tis2}
\end{align}
Let us introduce 
\begin{align*}
h_n(w) &\eqdef (1/2n) \sum_{i=1}^{n} (\langle w,x_i \rangle - y_i)^2\\
\widehat \xi^k &\eqdef \widehat{\varepsilon}^k - \nabla h_n(\widehat{w}^k) + (\langle w,x_{i(k)} \rangle - y_{i(k)})x_{i(k)},
\end{align*}
so that 
\eqref{tis2.0} and \eqref{tis2}  can be rewritten as
\begin{equation}\label{tis4}
\widehat v^k \eqdef \frac{\widehat w^k - \widehat w^{k+1}}{\alpha_k \gamma_k} - \xi_k - \nabla {h_n}(\widehat w^k) \in \lambda_n \partial R(\widehat z^k).
\end{equation} 

The missing block to conclude the proof of Theorem \ref{thm-ident-sgd} is then the next proposition whose proof is in the supplementary material.
\begin{prop}\label{P:cv algo}
Let $n \in \NN$, $\lambda_n \in ]0,+\infty[$, and let $(\widehat w^k)_{k \in \NN}$ be generated by the \eqref{algorithm} algorithm under assumption \eqref{H:hypothesis on algorithm}. 
Then
$(\widehat z^k,\widehat v^k)$ converges almost surely to $(\widehat w_{\lambda_n,n},\widehat \eta_{\lambda_n,n})$, as $k \to +\infty$.
\end{prop}


We can now complete the proof of Theorem \ref{thm-ident-sgd} as follows.
In  light of Proposition \ref{P:cv algo}, we deduce that there exists $K \in \NN$ such that for all $k \geq K$,
\begin{equation*}
\max\{ \norm{\widehat w_{\lambda_n,n} - \widehat z^k} , \norm{\widehat \eta_{\lambda_n,n} - \widehat v^k} \} \leq \delta/2 \quad a.s.
\end{equation*}
Without loss of generality, we can assume that the limit of the algorithm  is the $\widehat w_{\lambda_n,n}$ appearing in \eqref{tis1.5}. 
The above inequality, combined with  \eqref{tis1.5}, allows us to use Proposition \ref{T:mirror strat identification}, 
and this proves Theorem \ref{thm-ident-sgd}.

%% file: S-numerics.tex

\begin{figure*}[!htb]\centering
\begin{tabular}{@{}c@{\hspace{10mm}}c@{}}
\includegraphics[width=.43\linewidth]{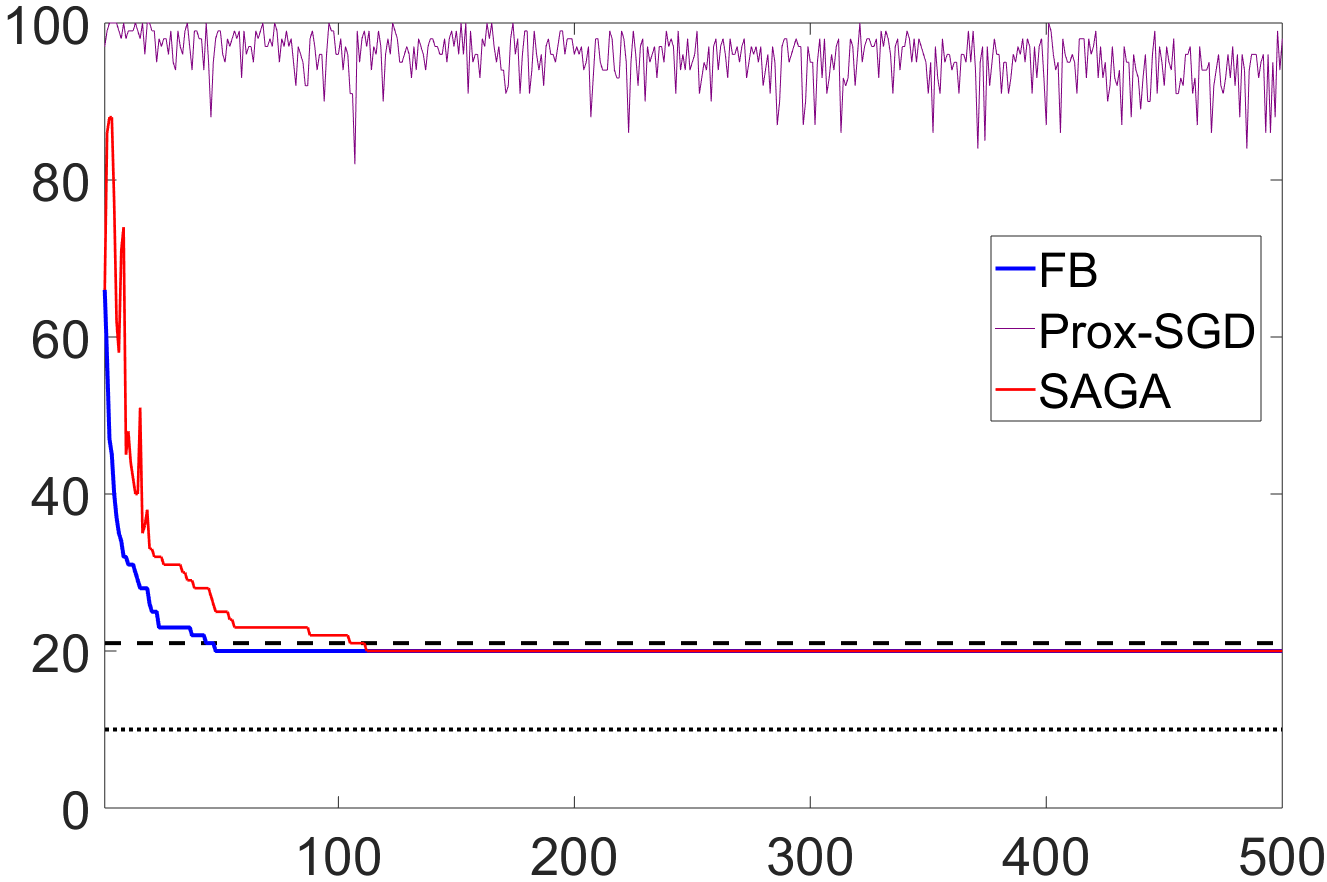} &
\includegraphics[width=.43\linewidth]{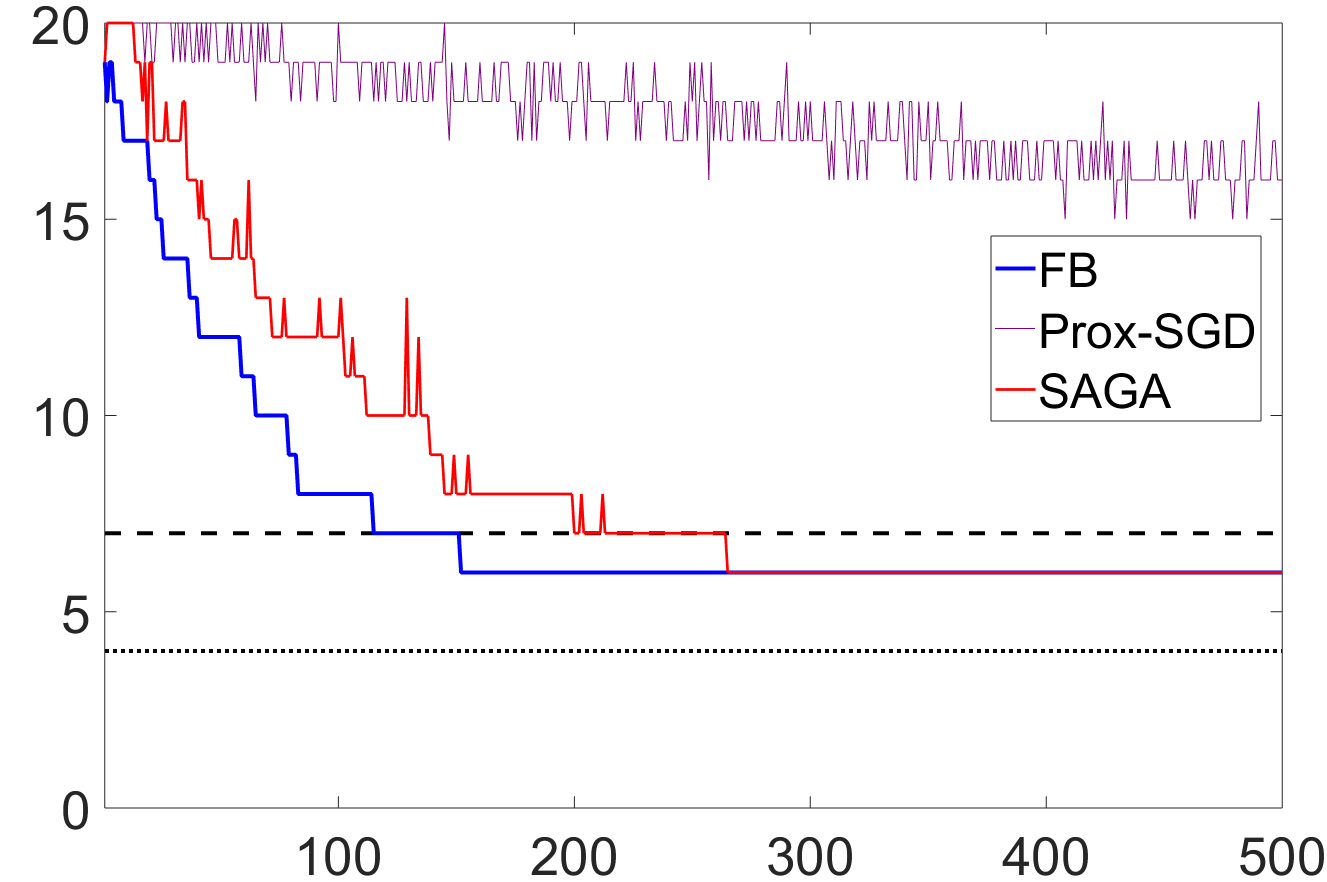}\\
$R=\norm{\cdot}_1$ & $R=\norm{\cdot}_*$
\end{tabular}
\caption{Evolution of $R_0( M_{\widehat{w}^k} )$ along the batches of the FB (blue), Prox-SGD (purple) and SAGA  (red) algorithms, applied to solve~\eqref{eq-empirical-risk-bis}. 
The black dotted line (resp. black dashed line) indicates the value $R_0( M_{w_0})$ (resp. the value of $R_0(\Jj_{R^*}(M^*_{\eta_0}))$ ); these are the dimensions of the two extreme strata.}
\label{F:comparison SGD SAGA FB}
\end{figure*}

\begin{figure*}[!htb]\centering
\begin{tabular}{@{}c@{\hspace{10mm}}c@{}}
\includegraphics[width=.43\linewidth]{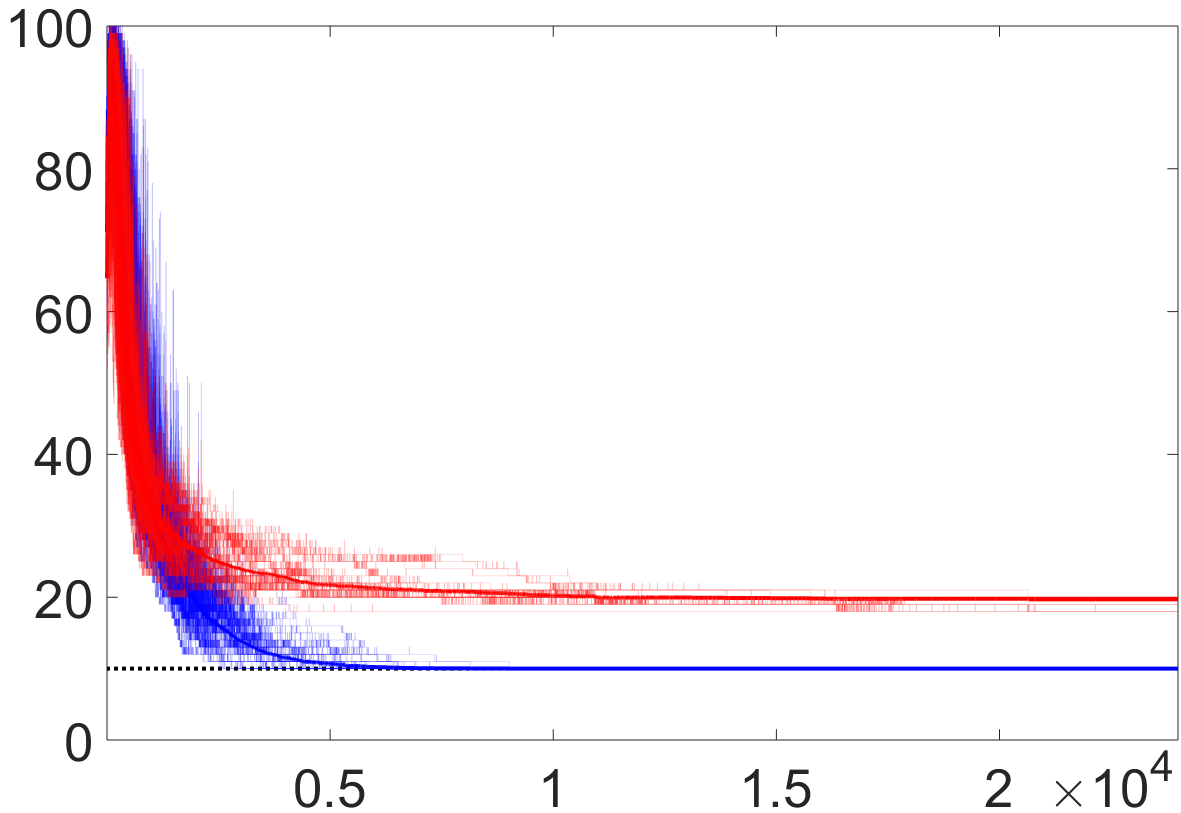} &
\includegraphics[width=.43\linewidth]{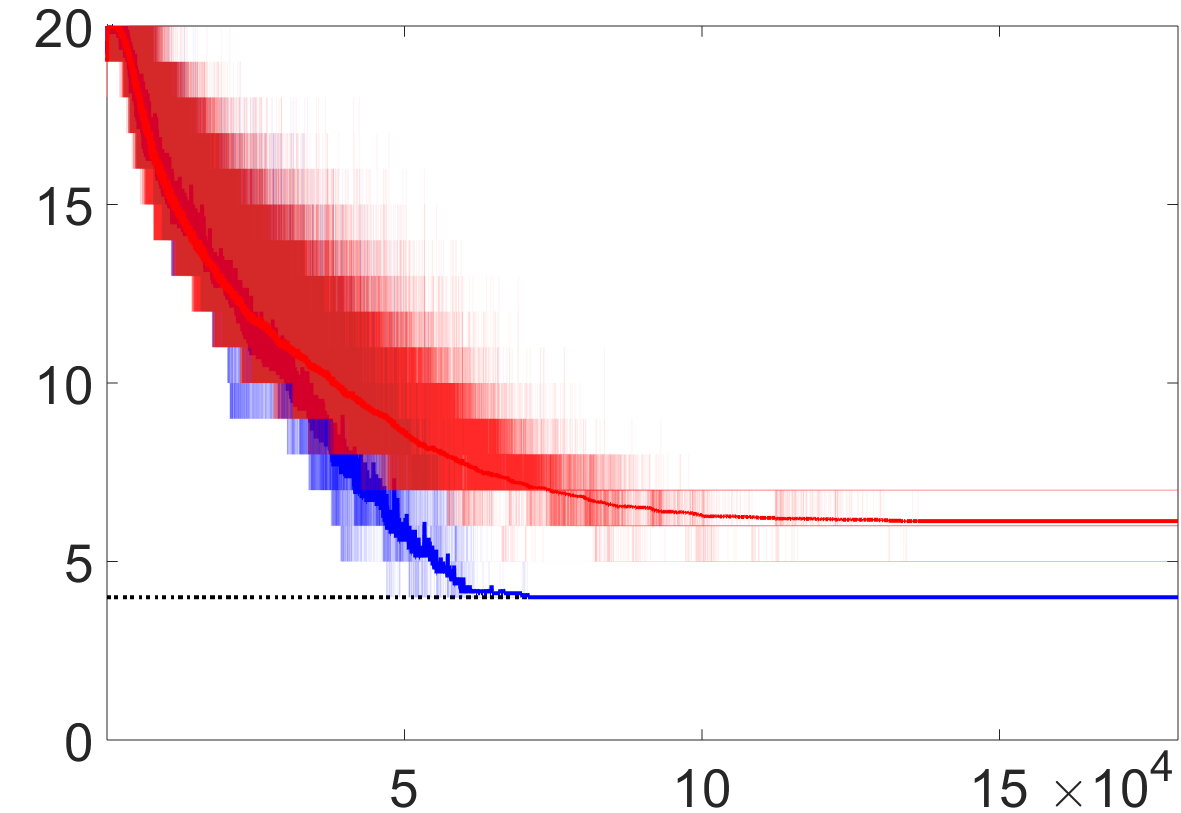}\\
$R=\norm{\cdot}_1$ & $R=\norm{\cdot}_*$
\end{tabular}
\caption{\label{F:various plots for SAGA}
 Evolution of $R_0(M_{\widehat{w}^k})$ along the iterations of SAGA, for problems where $ R_0(\Jj_{R^*}(M^*_{\eta_0})) =  R_0( M_{w_0}) + \delta$. 
Blue trajectories correspond to problems for which $\delta =0$, and for red trajectories  $\delta=10$ (left) or $3$ (right). The thick lines correspond to averaged trajectories.
 The black dotted line indicates the value $R_0( M_{w_0})$.
}
\end{figure*}

\section{Numerical illustrations for sparse/low-rank regularization}
\label{sec-numerics}

We give some numerical illustrations of our model consistency results for two popular regularizers: the $\ell_1$-norm and the nuclear norm. We generate random problem instances and control the low-complexity of the primal-dual pair of strata $(M_{\weight_0},\Jj_{R^*}(M^*_{\eta_0}))$. The low-complexity of a strata $M_{\weight}$ (i.e., the level of low-complexity of $\weight$) is measured by
\begin{align*}
R_0(M_\weight) &\eqdef \norm{\weight}_0 \quad \text{for $R=\norm{\cdot}_1$} ,\\
R_0(M_\weight) &\eqdef \rank(\weight) \quad \text{for $R=\norm{\cdot}_*$}.
\end{align*}
Observe that $R_0$ is well defined, since it does not depend of the choice of $w$ in the strata (see Example \ref{Ex:strata}).
\paragraph{Setup.}
The instances are randomly generated as follows.
For $R=\norm{\cdot}_1$, $\weight_0$ is drawn randomly among sparse vectors with sparsity level $R_0(M_{\weight_0}) = \norm{\weight_0}_0=s$, and we take $(p,n,s,\lambda)=(100,50,10,0.2)$. 
For $R=\norm{\cdot}_*$, $\weight_0$ is drawn randomly among low-rank matrices with rank $R_0(M_{\weight_0}) = \rank(\weight_0)=s$, and we take $(p,n,s,\lambda) = (20\times 20,300,4,0.03)$. 
The features $(x_i)_{i=1}^{n}$ are drawn at random in $\RR^{p}$ with i.i.d.\;entries from a zero-mean standard Gaussian distribution.
We take $y_i$ as $\langle w_0,x_i \rangle$, to which we add a zero-mean white Gaussian noise with standard deviation $10^{-2}$. 
We compute $\eta_0$ with an interior point solver, from which we deduce the upper-bound $R_0(\Jj_{R^*}(M^*_{\eta_0}))$.

%


\paragraph{FB vs.\;Prox-SGD vs.\;SAGA.}
First, we compare the deterministic Forward-Backward (FB) algorithm, the Prox-SGD method and SAGA on a simple instance of~\eqref{eq-empirical-risk-bis}. All algorithms are run with $\alpha_k \equiv 1$, and we take $\gamma_k \equiv 1.8/ L_n$ for the FB algorithm, $\gamma_k = 10/(k + 3\times 10^4)$, and $\gamma_k \equiv 1/(3L_n')$ for SAGA, where $L_n$ is defined in \eqref{H:hypothesis on algorithm} and $L_n' \eqdef \max_{i} \norm{x_i}^2$. 
Figure~\ref{F:comparison SGD SAGA FB} depicts the evolution of $R_0(M_{ \widehat w^k })$ while running these three algorithms on \eqref{eq-empirical-risk-bis}.
At each iteration, FB visits all the data at once, while Prox-SGD and SAGA need only one data. To fairly compare  these three algorithms, we plot only the iterates at every batch (i.e. all iterates for FB, and one every $n$ iterates for the stochastic algorithms).

As expected, the two stochastic algorithms exhibit an oscillating behaviour.
But for SAGA, these oscillations are damped quickly, and the support of $\widehat{w}^k$ stabilizes after a finite number of iterations.
On the contrary, Prox-SGD suffers from constant variations of the support, and is unable to generate iterates with a sparse support. 
Another observation is that FB and SAGA identify a support which is larger than the one of $w_0$ but below the extended one governed by $\eta_0$,    which is in agreement with Theorem~\ref{thm-ident-sgd}. 
A natural question is then: if we replace $w_0$ by another low-complexity vector, and consider other data, what can be said about the complexity of the obtained solution? This is discussed next.

\paragraph{Randomized experiments for SAGA.}

We now focus on SAGA, and look at the strata that its iterates can identify.
For the $\ell_1$ norm (resp. nuclear norm), we draw 1000 (resp. 200) realizations of $(w_0,(x_i,y_i)_{i=1}^n)$ exactly as before.
For each realization, we compute $\eta_0$ with high precision by using a solver.
We then select among the  realizations those for which $R_0(\Jj_{R^*}(M^*_{\eta_0}))$ belongs exactly to $\{10,20 \}$ for the $\ell_1$ norm (resp. to $\{4,7\}$ for the nuclear norm), and we apply the SAGA algorithm to these.
The evolution of $R_0(M_{ \widehat w^k })$ in these cases are plotted in Figure~\ref{F:various plots for SAGA}.

We see that for the realizations for which $R_0(\Jj_{R^*}(M^*_{\eta_0})) = R_0(M_{w_0})$ (the blue curves), the algorithm indeed identifies in finite time the stratum where  $w_0$ belongs. 
Otherwise, we see that the algorithm often identifies a stratum of the same dimension as that of $R_0(\Jj_{R^*}(M^*_{\eta_0}))$, or sometimes smaller, but which is always larger than $M_{w_0}$.
These observations are consistent with  the predictions of Theorem~\ref{thm-ident-sgd}.


%% file: S-conclusion.tex

\section{Conclusion}

In this paper, we provided a fine and unified analysis for studying model stability/consistency, when considering empirical risk minimization with a mirror-stratifiable regularizer, and solving it with a stochastic algorithm.
We showed that, even in the absence of the irrepresentable condition, the low-complexity of an approximate empirical solution remains controlled by a dual certificate. Moreover, we proposed a general algorithmic framework in which stochastic algorithms inherit almost surely finite activity identification.


%% file: S-annex.tex

\twocolumn

\section{Supplementary material}

This is the supplementary material for the paper \textit{Model Consistency for Learning with Mirror-Stratifiable Regularizers}. It contains the detailed proofs of the propositions \ref{P:cv problem} and \ref{P:cv algo} of which are, as explained in Section \ref{S:sketch proof}, the building blocks of our two main results (Theorems~1 and 2). The supplementary is structured in three sections: Section\;\ref{SS:prequel to the proof of T:identification problem proba} gathers key technical lemmas; Section\;\ref{SS:proof of Theorem 1 - cv of primal variable} presents the proof of Proposition \ref{P:cv problem}; 
Section\;\ref{sec-proof-algo} presents the one of\;Proposition\;\ref{P:cv algo}.

We use the same notations here as introduced at the beginning of Section \ref{S:sketch proof}.
We also introduce what can viewed as the limit of~\eqref{eq-empirical-risk-bis} as $n \rightarrow +\infty$:
\eql{\label{eq-limitting-risk-lagr}
	\weight_\la \in \uArgmin{\weight \in \RR^p}  \la R(\weight) + \frac{1}{2}\dotp{C \weight}{\weight} - \dotp{u}{\weight}.
}
For any positive semi-definite matrix $A$, we also note the seminorm $\norm{\cdot}_A=\sqrt{\dotp{A\cdot}{\cdot}}$.

\subsection{Useful technical lemmas}\label{SS:prequel to the proof of T:identification problem proba}

Here we present a few technical lemmas.
The first gives us some control on how $\widehat{C}_n$ converges to $C$ (resp. $\widehat{u}_n$ converges to $u$) when the amount of data $n$ tends to $+\infty$, and the second provides us with some essential compactness on these sequences.
The third  provides us an important variational characterization of the set to which belongs $\eta_0$.
Finally, the last Lemma gives a useful estimate between $\widehat w_{\lambda,n}$ and $w_\lambda$.

\begin{lem}\label{L:properties of C_n in stochastic}
If $\lambda_n \sqrt{n/\log \log n} \to +\infty$ and $\EE\left[\vert \yp \vert^4\right] + \EE\left[\norm{\xp}^4\right] <+\infty$, then the following holds almost surely:
\begin{enumerate}[label=(\roman*)]
	\item\label{L:properties of C_n in stochastic:speed} $\max\{ \Vert \widehat u_n - u \Vert , \Vert \widehat C_n - C \Vert \} = o(\lambda_n)$,
	\item\label{L:properties of C_n in stochastic:equal range} for $n$ large enough, $\Im \widehat{C}_n = \Im C$,
	\item\label{L:properties of C_n in stochastic:cv pseudoinverse} $\widehat{C}_n^\dagger \rightarrow C^\dagger$ as $n \to +\infty$.
\end{enumerate}
\end{lem}

\begin{proof}
It can be seen (use the Young inequality) that
\begin{eqnarray*}
 \EE\left[\norm{\bold{xy}}^2\right]  & = & \EE\left[\vert \yp \vert^2 \norm{\xp}^2\right]\\ 
& \leq &  
\frac{1}{2}\EE\left[\vert \yp \vert^4\right] + \frac{1}{2}\EE\left[\norm{\xp}^4\right] <+\infty \\
 \text{ and } \EE\left[\norm{\bold{xx}^\top}^2\right] & =& \EE\left[\norm{\bold x}^4\right] < + \infty.
\end{eqnarray*}
We are then in a position to invoke the law of iterated logarithm \citep[Proposition~2.26]{Vaa98} to obtain that, with probability 1,
$$r_n\eqdef\max\{ \Vert \widehat u_n - u \Vert , \Vert \widehat C_n - C \Vert \} = O\pa{n^{-1/2}\sqrt{{\log \log n}}}.$$
Our assumption that $\lambda_n \sqrt{n/\log \log n} \to +\infty$ then entails item \ref{L:properties of C_n in stochastic:speed}.

We now turn to item \ref{L:properties of C_n in stochastic:equal range}.
Consider $w \in \Ker C$; it verifies by definition $\EE_\rho[ \xp \langle \xp,w \rangle ] =  0$.
By taking the scalar product of this equality with $w$, we see that $(\forall x {~\sim~} \rho)$, \  $ \PP( \langle x,w\rangle =0)=1.$
Let ($w_1,...,w_d$) be a basis of $\Ker C$, where $d=\dim(\ker C)$. 
Then we  deduce that $(\forall x {~\sim~} \rho)$, \  $\PP((\forall i\in\{1,\ldots,d\}) \ \langle x,w_i\rangle =0)=1$.
In other words, $x \in (\Ker C)^\perp$ a.s., or, equivalently:
\begin{equation}\label{irc2}
(\forall x \underset{\tiny i.i.d.}{\sim} \rho) \quad \PP(x \in \Im C)=1.
\end{equation}
Now, observe that $\Im \widehat C_n = \Im\pa{\{x_i\}_{i=1}^n}$, so the following implication holds:
\begin{equation}\label{irc3}
\left[ (\forall i\in\{1,\ldots,n\}) \ x_i \in \Im C \right] \Rightarrow \Im \widehat C_n \subset \Im C.
\end{equation}
Since the $x_i$ are drawn i.i.d. from $\rho$, and are in finite number, we can combine \eqref{irc2} and \eqref{irc3} to obtain that 
\begin{eqnarray*}
\PP(\Im \widehat C_n \subset \Im C) & \geq & \PP((\forall i\in\{1,\ldots,n\}) \ x_i \in \Im C) \\
&=& \prod_{i=1}^n \PP(x_i \in \Im C) = 1 .
\end{eqnarray*}
We deduce then that $\Im \widehat C_n \subset \Im C$ a.s., from which we get that $\rank \widehat{C}_k \leq \rank C$ a.s. This, together with lower semi-continuity of the rank, yields that with probability 1,
\begin{eqnarray*}
\rank (C) \leq \liminf\limits_{k \to +\infty} \rank (\widehat{C}_k) 
 & \leq &  \limsup\limits_{k \to +\infty} \rank (\widehat{C}_k) \\
 &  \leq  & \rank C \ \,
\end{eqnarray*}
meaning that $\rank (\widehat{C}_k) \to \rank (C)$ a.s.
Because the rank takes only discrete values, this means that $ \rank \widehat{C}_k = \rank C$ a.s. for all $k$ large enough.
We can then trivially deduce from the inclusion $\Im \widehat{C}_k \subset \Im C$ a.s., that the equality $\Im \widehat{C}_k = \Im C$ holds a.s. for $k$ large enough.

Assertion~\ref{L:properties of C_n in stochastic:cv pseudoinverse} follows from \ref{L:properties of C_n in stochastic:equal range} and \citep[Theorem~3.3]{Ste77}.
\end{proof}


\begin{lem}\label{L:boundedness primal pb}
Assume that \eqref{H:hypothesis on model} holds, and
\begin{itemize}
	\item $\Im \widehat{C}_n = \Im C$ for $n$ large enough,
	\item $\sup_{n \in \NN}\widehat{C}_n^\dagger < +\infty$.
\end{itemize}
Then,  the sequences $(\widehat w_{\lambda_n,n})_{n\in\NN}$ and $( w_{\lambda_n})_{n\in\NN}$ are bounded. 
\end{lem}

\begin{proof}

Introduce $f_{\lambda}(w) \eqdef R(w) + (1/2\lambda) \norm{C w - u}^2_{C^\dagger}$ and $ f_{\lambda,n}(w)\eqdef R(w) + (1/2\lambda) \norm{\widehat C_n w - \widehat u_n}^2_{\widehat C_n^\dagger}$
which, by definition, verify 
\begin{equation*}
w_{\lambda_n} \in \Argmin f_{\lambda_n} \text{ and } \widehat w_{\lambda_n,n} \in \Argmin f_{\lambda_n,n}.
\end{equation*}
Define $\supla \eqdef \sup_{n} \lambda_n > 0$, and use the optimality of $\widehat w_{\lambda_n,n}$ to derive
\begin{equation*}\label{eq:slevfn1}
f_{\supla,n}(\widehat w_{\lambda_n,n}) \leq
f_{\lambda_n,n}(\widehat w_{\lambda_n,n})
\leq f_n(w_0)
\end{equation*}
By making use of  Lemma~\ref{L:properties of C_n in stochastic}.\ref{L:properties of C_n in stochastic:cv pseudoinverse} and  Lemma~\ref{L:properties of C_n in stochastic}.\ref{L:properties of C_n in stochastic:speed}, we have the bound
\begin{eqnarray*}
f_n(w_0) 
& \leq & R(w_0) + \frac{\norm{\widehat C_n^\dagger}}{2\lambda_n} \norm{\widehat C_n w_0 - \widehat u_n}^2, \\
& \leq & R(w_0) + O\pa{\frac{\norm{\widehat C_n -C} + \norm{u - \widehat u_n}}{\lambda_n}}^2, \nonumber\\
& \leq & R(w_0) + o(1).
\end{eqnarray*}
We can make a similar reasoning on the sequence $( w_{\lambda_n})_{n\in\NN}$, and deduce that
\begin{eqnarray}
f_{\supla}(w_{\lambda_n})  \leq R(w_0) + o(1), \label{eq:slevfn2.1} \\ 
\text{ and }
f_{\supla,n}(\widehat w_{\lambda_n,n}) \leq R(w_0) + o(1). \label{eq:slevfn2.2}
\end{eqnarray}

To prove the boundedness of $(\widehat w_{\lambda_n,n})_{n\in\NN}$ and $( w_{\lambda_n})_{n\in\NN}$, we will  use arguments relying on  the notion of asymptotic or recession function; see~\cite[Definition 10.32]{BauComv2} for a definition.
Define $f_0(w) \eqdef R(w) + \iota_{\{u\}}(Cw)$, where $\iota_{\{u\}}$ is the indicator function\footnote{The indicator function $\iota_\Omega$ of a set $\Omega\subset \RR^p$ is by definition equal to $0$ when evaluated on $\Omega$, and $+\infty$ elsewhere.} of the singleton $\{u\}$.
The hypothesis \eqref{H:hypothesis on model} indicates that $\argmin f_0 = \{w_0\}$, so in particular $\argmin f_0$ is compact.
We can then invoke \cite[Proposition~3.1.2 and~3.1.3]{Auslender03} to deduce that $f_0^\infty(w) > 0$ for all $w \in \RR^p \setminus \{0\}$, where $f_0^\infty$ is  the recession function of $f_0$.
From the sum rule \cite[Proposition~2.6.1]{Auslender03}, we deduce that $f_0^\infty = R^\infty + (\iota_{\{u\}} \circ C)^\infty$.
Moreover, we know from \eqref{H:hypothesis on model} that $u \in \Im C$, so we can use \cite[Proposition~2.6.1]{Auslender03} to get $(\iota_{\{u\}} \circ C)^\infty = \iota_{\{0\}} \circ C = \iota_{\Ker C}$.
We deduce from all this that $R^\infty(w) > 0$ for all $w \in \Ker C \setminus \{0\}$, which can be equivalently reformulated as
\begin{equation}\label{eq:slevfn0}
\Ker R^\infty \cap \Ker C = \{0\}.
\end{equation}

Let us start with the boundedness of $( w_{\lambda_n})_{n\in\NN}$.
Combining~\cite[Proposition~2.6.1]{Auslender03}, \cite[Example~2.5.1]{Auslender03} and the fact that $u \in \Im C$, the recession function of $f_{\supla}$ reads $f_{\supla}^\infty(w)=R^\infty(w)$ if $w \in \ker C$ and $+\infty$ otherwise. 
Thus, \eqref{eq:slevfn0} is equivalent to $f_{\supla}^\infty(w) > 0$ for all $w \neq 0$.
This is equivalent to saying that $f_{\supla}$ is level-bounded (see \cite[Proposition~3.1.3]{Auslender03}), from which we deduce boundedness of $( w_{\lambda_n})_{n\in\NN}$ via \eqref{eq:slevfn2.1} and \eqref{eq:slevfn2.2}.

We now turn on $(\widehat w_{\lambda_n,n})_{n\in\NN}$.
 We write $\widehat u_n = C \widehat p_n$ since $\widehat u_n \in \Im \widehat C_n \subset \Im C$. We first observe that \eqref{eq:slevfn2.1} and \eqref{eq:slevfn2.2} can be rewritten as:
\begin{equation*}
\frac{1}{2\supla} \norm{\widehat C_n (\widehat w_{\lambda_n,n} - \widehat p_n)}^2_{\widehat C_n^\dagger} + R(\widehat w_{\lambda_n,n}) 
 \leq R(w_0) + o(1).
\end{equation*}
Let $V_n \diag(s_{n,i}) V_n^\top$ be a (reduced) eigendecomposition of $\widehat C_n$. By our assumptions, we have $\infs \eqdef \inf_{n,1 \leq i \leq r} s_{n,i} = \bpa{\sup_n \norm{\widehat C_n}}^{-1} > 0$. In addition, the columns of $V_n$ form an orthonormal basis of $\Im C$ for $n$ large enough. Thus, for all such $n$, we have
\begin{eqnarray*}
& & \infs \norm{\Proj_{\Im C}(\widehat w_{\lambda_n,n}-\widehat p_n)}^2 \\
& = &  \infs \norm{V_n^\top (\widehat w_{\lambda_n,n}-\widehat p_n)}^2 \\
&\leq & \sum_{i=1}^r s_{n,i} \abs{\dotp{v_{n,i}}{\widehat w_{\lambda_n,n}-\widehat p_n}}^2 \\
&= & \dotp{\widehat C_n(\widehat w_{\lambda_n,n}-\widehat p_n)}{\widehat w_{\lambda_n,n}-\widehat p_n} \\
&= & \norm{\widehat C_n (\widehat w_{\lambda_n,n} - \widehat p_n)}^2_{\widehat C_n^\dagger} .
\end{eqnarray*}
Altogether, we get the bound
\[
\frac{\infs}{2\supla} \norm{\Proj_{\Im C} (\widehat w_{\lambda_n,n} - \widehat p_n)}^2 + R(\widehat w_{\lambda_n,n}) \leq R(w_0) + o(1)
\]
for $n$ sufficiently large. Arguing as above, the recession function of $g \eqdef \frac{\infs}{2\supla}\norm{\cdot - \widehat p_n}^2 \circ \Proj_{\Im C} + R$ is again $g^\infty(w) = R^\infty(w)$ if $w \in \ker C$ and $+\infty$ otherwise, independently of $\widehat p_n$\footnote{This reflects the geometric fact that the recession function is unaffected by translation of the argument.}. Our assumption plugged into \cite[Proposition~3.1.3]{Auslender03} entails that $g$ is level-bounded and thus boundedness for $(\widehat w_{\lambda_n,n})_{n\in\NN}$.
\end{proof}

\begin{lem}\label{L:limit solutions belong to partial R}
Assume that \eqref{H:hypothesis on model} holds.
Then
\begin{equation*}
 \uArgmin{\eta \in \Im C} R^*(\eta) - \dotp{C^\dagger u}{\eta} = \partial R(w_0) \cap \Im C.
\end{equation*}
\end{lem}

\begin{proof}
Using \citep[Proposition~13.23 \& Theorem~15.27]{BauComv2}, one can check that problem 
\[
\umin{\eta \in \Im C} R^*(\eta) - \dotp{C^\dagger u}{\eta}
\]
is the Fenchel dual of \eqref{Eq:primal expected problem}. Moreover, $(w^\star,\eta^\star)$ is a primal-dual (Kuhn-Tucker) optimal pair if and only if
\[
\begin{pmatrix}
w^\star \\
\eta^\star
\end{pmatrix} \in
\begin{pmatrix}
C^\dagger u + \ker C \\
\partial R(w^\star) \cap \Im C
\end{pmatrix} .
\]
As we assumed in \eqref{H:hypothesis on model} that $w_0$ is the unique minimizer of \eqref{Eq:primal expected problem}, the claimed identity follows.
\end{proof}

\begin{lem}\label{L:estimate primal pb}
Let $n\in \NN$ and assume that $\Im \widehat C_n \subset  \Im C$.
Denote $r_n\eqdef \max\ens{\norm{\widehat u_n - u}, \norm{\widehat C_n - C}}$.
Then,
\begin{equation*}
\norm{C(\widehat w_{\lambda,n} - w_\lambda)} \leq (\norm{C}\norm{C^\dagger})^{1/2} (1+\norm{\widehat w_{\lambda,n}})r_n.
\end{equation*}
\end{lem}

\begin{proof}
The first-order optimality conditions for both $\widehat w_{\lambda,n}$ and $w_\lambda$ yield
\[\left\{
\begin{array}{rl}
  0 & \in \lambda \partial R(\widehat w_{\lambda,n}) + \widehat C_n \widehat w_{\lambda,n} - \widehat u_n \\
0 & \in \lambda \partial R( w_{\lambda}) +  C  w_{\lambda} -  u.
\end{array}
\right.\]
In view of monotonicity of $\partial R$, we deduce that 
\begin{equation*}
0 \leq \dotp{\widehat u_n - u + C w_\lambda - \widehat C_n \widehat w_{\lambda,n}}{\widehat w_{\lambda,n}  - w_\lambda}.
\end{equation*}
Rearranging the terms, we get
\begin{eqnarray}
& & \dotp{C(\widehat w_{\lambda,n} - w_\lambda)}{\widehat w_{\lambda,n}  - w_\lambda} \label{epp1} \\
& \leq & \dotp{\widehat u_n - u+ (C- \widehat C_n)\widehat w_{\lambda,n}}{\widehat w_{\lambda,n}  - w_\lambda} \nonumber. 
\end{eqnarray}
By virtue of standard properties of the Moore-Penrose pseudo-inverse and the fact that $\widehat u_n - u$ and $C-\widehat C_n$ both live in $\Im C \supset \Im \widehat C_n$, we obtain
\begin{eqnarray*}
& & \dotp{C^\dagger (C\widehat w_{\lambda,n} - Cw_\lambda)}{C\widehat w_{\lambda,n}  - C w_\lambda} \\
& \leq & \dotp{C^\dagger(\widehat u_n - u+ (C- \widehat C_n)\widehat w_{\lambda,n})}{C\widehat w_{\lambda,n}  - Cw_\lambda}. 
\end{eqnarray*}
Applying the Cauchy-Schwarz and triangle inequalities, we arrive at
\begin{eqnarray*}
& &\norm{C\widehat w_{\lambda,n} - Cw_\lambda}_{C^\dagger} \\
& \leq & \norm{\widehat u_n - u}_{C^\dagger} + \norm{(C- \widehat C_n)\widehat w_{\lambda,n}}_{C^\dagger} \\
& \leq & \norm{C^\dagger}^{1/2}\pa{\norm{\widehat u_n - u} + \norm{C - \widehat C_n}\norm{\widehat w_{\lambda,n}}} \\
& \leq & \norm{C^\dagger}^{1/2} (1+\norm{\widehat w_{\lambda,n}})r_n.
\end{eqnarray*}
On the left side of this inequality, we exploit the fact that $\norm{C}^{-1}$ is the smallest nonzero eigenvalue of $C^\dagger$ on $\Im(C)$ to conclude
\begin{eqnarray*}
\norm{C\widehat w_{\lambda,n} - Cw_\lambda}_{C^\dagger}
 & \leq & \norm{C}^{1/2}\norm{C\widehat w_{\lambda,n} - Cw_\lambda}_{C^\dagger} \\
& \leq & (\norm{C}\norm{C^\dagger})^{1/2} (1+\norm{\widehat w_{\lambda,n}})r_n .
\end{eqnarray*}
\end{proof}

\subsection{Proof of Proposition  \ref{P:cv problem}}\label{SS:proof of Theorem 1 - cv of primal variable}

\paragraph{Convergence of the primal variable.}
To lighten notations, we will write $\widehat w_n \eqdef \widehat w_{\lambda_n,n}$.
From Lemma~\ref{L:boundedness primal pb} we know that $(\widehat w_{n})_{n \in \NN}$ is bounded a.s., so it admits a cluster point, say $w^\star$.
Let $\widehat w_n$ be a subsequence (we do not relabel for simplicity) converging a.s. to $w^\star$.
Now, let $\varepsilon_n \eqdef \widehat u_{n} - \widehat C_{n} w_0$, for which we know that both $\varepsilon_n$ and $ \varepsilon_n / \lambda_n$ are $o(1)$, thanks to Lemma~\ref{L:properties of C_n in stochastic}\ref{L:properties of C_n in stochastic:speed} and the fact that $u = C w_0$.
From the optimality of $\widehat w_n$, we obtain
\begin{eqnarray*}
& & \lambda_n R(\widehat w_n) + \frac{1}{2}\langle \widehat C_n \widehat w_n, \widehat w_n \rangle - \langle \widehat u_n , \widehat w_n \rangle \\
& \leq & \lambda_n R( w_0) + \frac{1}{2}\langle  \widehat C_n w_0, w_0 \rangle - \langle  \widehat u_n , w_0 \rangle,
\end{eqnarray*}
which can be equivalently rewritten as
\begin{eqnarray}\label{ident3} 
& & \frac{1}{2}\langle \widehat C_n  (\widehat w_n - w_0), \widehat w_n - w_0 \rangle - \langle \widehat w_n - w_0, \varepsilon_n \rangle \hspace*{1cm} \\
& \leq & \lambda_n (R(w_0) - R(\widehat w_n)) \nonumber.
\end{eqnarray}
Passing to the limit in~\eqref{ident3} and using the fact that $R$ is bounded from below, we obtain 
\begin{equation*}
\dotp{C( w^\star - w_0 )}{w^\star - w_0} =0 \text{ a.s. },
\end{equation*}
 or equivalently, that $Cw^\star = Cw_0 = u$ a.s. since $C$ is positive semi-definite.
In addition, as $\widehat C_n$ is also positive semi-definite, so we can rewrite \eqref{ident3} as
\begin{equation}\label{ident4}
R(\widehat w_n)  \leq R(w_0) + \dotp{\widehat w_n - w_0}{\frac{\varepsilon_n}{\lambda_n}}.
\end{equation}
Passing to the limit in \eqref{ident4}, using lower-semicontinuity of $R$ and that $\varepsilon_n / \lambda_n = o(1)$ a.s., we arrive at 
\begin{equation*}
R(w^\star) \leq \liminf_n R(\widehat w_n) \leq \limsup_n R(\widehat w_n) \leq R(w_0) \quad a.s.
\end{equation*}
Clearly $R(w^\star) \leq R(w_0)$ and $w^\star$ obeys the constraint $Cw^\star = u$, which implies that $w^\star$ is a solution of \eqref{Eq:primal expected problem} a.s.
But since this problem has a unique solution, $w_0$, by assumption~\eqref{H:hypothesis on model}, we conclude that $w^\star=w_0$ a.s.
This being true for any a.s. cluster point means that $\widehat w_n \to w_0$ as $n \to +\infty$ a.s.

\paragraph{Convergence of the dual variable.}

Here we omit systematically mentioning  that the bounds and convergence we obtain hold almost surely.

It can be verified, using for instance \citep[Proposition~13.23 \& Theorem~15.27]{BauComv2}, that the Fenchel dual problem of~\eqref{eq-empirical-risk-bis} is 
\begin{equation}\label{E:dual empirical penalized pb}
\ens{\widehat \eta_{\lambda,n}} \eqdef \uArgmin{\eta \in \Im \widehat C_n} R^*(\eta) + \frac{\lambda}{2} \langle \widehat C_n^\dagger \eta ,\eta \rangle - \langle \widehat C_n^\dagger \widehat u_n , \eta \rangle.
\end{equation}
For any fixed $\lambda >0$, we also introduce its limit problem\footnote{By Lemma~\ref{L:properties of C_n in stochastic}, we indeed have $C_n^\dagger \to C^\dagger$ a.s. under our hypotheses.}, as $n \to +\infty$ (which is the dual of \eqref{eq-limitting-risk-lagr}):
\begin{equation}\label{E:dual expected penalized pb}
\ens{\eta_{\lambda}} \eqdef \uArgmin{\eta \in \Im  C} R^*(\eta) + \frac{\lambda}{2} \langle  C^\dagger \eta ,\eta \rangle - \langle  C^\dagger  u , \eta \rangle.
\end{equation}
Both problems are strongly convex thanks to positive semi-definiteness of $\widehat C_n$ and $C$, hence uniqueness of the corresponding dual solutions $\widehat \eta_{\lambda,n}$ and $\eta_{\lambda}$. Moreover, from the primal-dual extremality relationships, see \citep[Proposition~26.1.iv.b]{BauComv2}, $\widehat \eta_{\lambda,n}$ and $\eta_{\lambda}$ can be recovered from the corresponding primal solutions as 
\begin{equation}\label{ident5}
\widehat \eta_{\lambda,n} \eqdef \frac{\widehat u_n - \widehat C_n \widehat w_{\lambda,n}}{\lambda} \qandq \eta_\lambda \eqdef \frac{u - C w_\lambda}{\lambda}.
\end{equation}

In what follows, we  prove  that $\widehat \eta_n$  converges to $\eta_0$ when $n \to +\infty$.
To lighten notation, we will denote $r_n\eqdef\max\{\Vert \widehat u_n - u \Vert , \Vert \widehat C_n - C \Vert \}$, and note $\widehat \eta_n = \widehat \eta_{\lambda_n,n}$.
We have
\begin{equation}\label{ident5.1.1}
\norm{\widehat \eta_n -\eta_0} \leq \norm{\widehat \eta_n - \eta_{\lambda_n}} + \norm{\eta_{\lambda_n} - \eta_0}.
\end{equation}
By using \eqref{ident5} and the definition of $r_n$, we  write
\begin{eqnarray*}
\norm{\widehat \eta_n - \eta_{\lambda_n}}  & = &
\normb{\frac{\widehat u_n - u }{\lambda_n} + \frac{C w_{\lambda_n} - \widehat C_n \widehat w_n}{\lambda_n}} \\
& \leq & O\pa{\frac{r_n}{\lambda_n}} + \normb{\frac{C w_{\lambda_n} - \widehat C_n \widehat w_n}{\lambda_n}}.
\end{eqnarray*}
The second term on the right hand side can also be bounded as
\begin{eqnarray*} 
\normb{\frac{C w_{\lambda_n} - \widehat C_n \widehat w_n}{\lambda_n}} & = &
\normb{\frac{C (w_{\lambda_n} -  \widehat w_n)}{\lambda_n}  + 
 \frac{ C \widehat w_n - \widehat C_n \widehat w_n}{\lambda_n}} \\
& \leq & \normb{\frac{C(w_{\lambda_n} -  \widehat w_n)}{\lambda_n}} + \norm{\widehat w_n} \frac{r_n }{\lambda_n} \\ 
 & = & O\pa{\frac{r_n}{\lambda_n}} ,
\end{eqnarray*}
where we used Lemma~\ref{L:estimate primal pb}, and Lemma~\ref{L:boundedness primal pb} with Lemma \ref{L:properties of C_n in stochastic} in the last inequality.
Combining the above inequalities with the fact that $r_n=o(\lambda_n)$ by Lemma \ref{L:properties of C_n in stochastic}.\ref{L:properties of C_n in stochastic:speed}, we obtain
\begin{equation}\label{ident5.1}
\norm{\widehat \eta_n - \eta_{\lambda_n}} = O\left(\frac{r_n}{\lambda_n} \right) \overset{n \to +\infty}{\longrightarrow} 0 .
\end{equation}
It remains now to prove that $ \eta_{\lambda}$ converges to $\eta_0 $ when $\lambda \to 0$.
To do so, we start by using optimality of $ \eta_\lambda$ and $\eta_0$ for problems \eqref{E:dual expected penalized pb} and \eqref{E:dual expected pb}, together with Lemma \ref{L:limit solutions belong to partial R}, to write
\begin{eqnarray}
& &  R^*(\eta_\lambda) + \frac{\lambda}{2} \langle C^\dagger \eta_\lambda,\eta_\lambda \rangle - \langle C^\dagger u, \eta_\lambda \rangle \label{ident7}  \\
& \leq & R^*(\eta_0) + \frac{\lambda}{2} \langle C^\dagger \eta_0,\eta_0 \rangle - \langle C^\dagger u, \eta_0 \rangle \nonumber \\
 & \leq & R^*(\eta_\lambda) + \frac{\lambda}{2} \langle C^\dagger \eta_0,\eta_0 \rangle - \langle C^\dagger u, \eta_\lambda \rangle \nonumber,
\end{eqnarray}
from which we deduce that
\begin{equation}\label{ident8}
\langle C^\dagger \eta_\lambda,\eta_\lambda \rangle \leq \langle C^\dagger \eta_0,\eta_0 \rangle.
\end{equation}
Since $\eta_\lambda \in \Im C = (\ker C^\dagger) ^\perp$ (see \eqref{E:dual expected penalized pb}), we can infer from \eqref{ident8} that $(\eta_\lambda)_{\lambda >0}$ is bounded.
Let $\eta^\star$ be any cluster point of this net, and let us verify that $\eta^\star$ must be equal to $\eta_0$.
First, passing to the limit in \eqref{ident8} shows that
\begin{equation}\label{ident8-1}
\langle C^\dagger \eta^\star,\eta^\star \rangle \leq \langle C^\dagger \eta_0,\eta_0 \rangle.
\end{equation}
Second, taking the limit in \eqref{ident7} and using lower semi-continuity of $R^*$, we get
\begin{eqnarray}\label{ident9}
& & R^*(\eta^\star) - \langle C^\dagger u , \eta^\star \rangle \\
 & \leq & \liminf_{\lambda \to 0} R^*(\eta_\lambda) + \frac{\lambda}{2} \langle C^\dagger \eta_\lambda,\eta_\lambda \rangle - \langle C^\dagger u, \eta_\lambda \rangle \hspace*{1cm} \nonumber \\
&  \leq &  \lim_{\lambda \to 0} R^*(\eta_0) + \frac{\lambda}{2} \langle C^\dagger \eta_0,\eta_0 \rangle - \langle C^\dagger u, \eta_0 \rangle \nonumber \\
& = & R^*(\eta_0) - \langle C^\dagger u, \eta_0 \rangle.\nonumber
\end{eqnarray}
From $\eta_\lambda \in \Im C$ we know that $\eta^\star \in \Im C$ as well, so we can then deduce from \eqref{ident9} and Lemma \ref{L:limit solutions belong to partial R}  that 
\begin{equation}\label{ident10}
\eta^\star \in \partial R(w_0) \cap \Im C.
\end{equation}
Putting together \eqref{ident8-1} and \eqref{ident10} shows that $\eta^\star$ is a solution of \eqref{E:dual expected pb}, hence $\eta^\star=\eta_0$ by uniqueness of $\eta_0$. This being true for any cluster point shows convergence of $\eta_\lambda$ to $\eta_0$.

\subsection{Proof of Proposition \ref{P:cv algo}  }\label{sec-proof-algo}

We use here the notations $h_n$ and $\widehat \xi^k$ introduced in Section \ref{S:sketch proof}, and we read directly from hypothesis \eqref{H:hypothesis on algorithm} that 
$\widehat{d}^k = \nabla h_n(\widehat{w}^k) + \widehat{\xi}^k$,  $\EE\left[\widehat{\xi}^k | \Ff_k \right] = 0, \ \EE\left[ \Vert \widehat{\xi}^k \Vert^2 | \Ff_k \right] \leq \sigma_k^2 \text{ and } \widehat{\xi}^k \text{ converges a.s. to } 0.$

Let us start by showing that $\widehat w^k$ converges to $\widehat w_{\lambda_n,n}$.
For this, let $w$ be any solution of \eqref{eq-empirical-risk-bis}.
We can write, using standard identities (e.g.~\cite[Corollary~2.14]{BauComv2}), that
\begin{eqnarray}\label{tis5}
& & \norm{\widehat w^{k+1} - w}^2 \\ 
&=& \norm{(1-\alpha_k)(\widehat w^k - w) + \alpha_k (\widehat z^k - w)}^2 \nonumber\\
\hspace*{-20cm}&=& (1-\alpha_k)\norm{\widehat w^k - w}^2 + \alpha_k\norm{\widehat z^k - w}^2 \nonumber \\
& & - \alpha_k(1-\alpha_k)\norm{\widehat z^k-\widehat w^k}^2 . \nonumber
\end{eqnarray}
Since $w$ is a  solution of \eqref{eq-empirical-risk-bis}, it is a fixed point for the operator $\Prox_{\lambda_n \gamma_k R} \circ (\Id - \gamma_k \nabla h_n)$ for any $k \in \NN$.
Use then the definition of $\widehat z^k$ together with the nonexpansiveness
 of the proximal mapping to obtain
\begin{align*}
\norm{\widehat z^k - w}^2
&\leq \norm{\widehat w^k - w + \gamma_k(\nabla h_n(w) - \nabla h_n(\widehat w^k) - \xi_k)}^2 \nonumber \\
&\leq \norm{\widehat w^k - w \Vert^2 + \gamma_k^2 \Vert \nabla h_n(w) - \nabla h_n(\widehat w^k) - \xi_k}^2 \nonumber \\
& \hspace*{0.3cm} + 2 \gamma_k \dotp{\widehat w^k - w}{\nabla h_n(w) - \nabla h_n(\widehat w^k) - \xi_k}. \nonumber
\end{align*}
Taking the conditional expectation w.r.t. $\mathcal{F}_k$ in the above inequality, and using the assumptions $\EE(\xi_k | \mathcal{F}_k) = 0$ and $\EE(\Vert \xi_k \Vert^2 | \mathcal{F}_k) \leq \sigma_k^2$, leads to 
\begin{eqnarray*}
& & \EE(\norm{\widehat z^k - w}^2 | \mathcal{F}_k) \\
&\leq & \norm{\widehat w^k - w}^2 + \gamma_k^2 \norm{\nabla h_n(w) - \nabla h_n(\widehat w^k)}^2 + \gamma_k^2 \sigma_k^2 \nonumber \\
& & \hspace*{0cm}+  2 \gamma_k \dotp{\widehat w^k - w}{\nabla h_n(w) - \nabla h_n(\widehat w^k)}.
\end{eqnarray*}
Since $\nabla h_n$ is $1/L$-cocoercive, we obtain
\[\begin{array}{lll}
&& \EE(\norm{\widehat z^k - w}^2 | \mathcal{F}_k) \\
&\leq & \norm{\widehat w^k - w}^2 + \gamma_k^2 \sigma_k^2 \\
& & - \gamma_k(2/L - \gamma_k)\norm{\nabla h_n(w) - \nabla h_n(\widehat w^k)}^2
\end{array}\]
After taking the conditional expectation in \eqref{tis5} and combining with the last inequality, we obtain
\begin{eqnarray*}
& & \EE(\norm{\widehat w^{k+1} - w }^2| \mathcal{F}_k) \\
&\leq &  \norm{\widehat w^k - w}^2  +  \alpha_k \gamma_k^2 \sigma_k^2 \\
& & - \gamma_k(2/L - \gamma_k)\norm{\nabla h_n(w) - \nabla h_n(\widehat w^k)}^2 \\
& & - \alpha_k(1-\alpha_k)\EE(\norm{\widehat z^k-\widehat w^k}^2| \mathcal{F}_k) .
\end{eqnarray*}
The inequality above means that $(\widehat w^k)_{k\in\NN}$ is a stochastic quasi-F\'ejer sequence, and hypothesis \eqref{H:hypothesis on algorithm} allows us to use invoke \cite[Proposition~2.3]{CombettesStochastic15}, from which we deduce that 
%
%
%
%
$(\widehat w^k)_{k \in \NN}$ is bounded a.s. 
Thus $\widehat w^k$ has a cluster point. 
Let $\bar{w}$ be a sequential cluster point of $(\widehat w^k)_{k \in \NN}$, and $\widehat w^{k}$ be a subsequence (that we do not relabel for simplicity) that converges a.s. to $\bar{w}$. Recalling \eqref{tis4} and \eqref{tis2}, and in view of assumption~\eqref{H:hypothesis on algorithm} and continuity of the gradient, we deduce that
\[
\widehat v^k \to -\nabla h_n(\bar{w}) \qandq \widehat z^k \to \bar{w} \quad a.s.
\]
Since $(\widehat z^k,\widehat v^k) \in \Gr(\lambda_n\partial R)$ and $\lambda_n\partial R$ is maximally monotone, we conclude that $0 \in \nabla h_n(\bar{w}) + \lambda_n\partial R(\bar{w})$, i.e., $\bar{w}$ is minimizer of \eqref{eq-empirical-risk-bis}. Since this is true for any cluster point, we invoke \cite[Proposition~2.3(iv)]{CombettesStochastic15} which yields that $\widehat w^k$ converges a.s. to a minimizer of \eqref{eq-empirical-risk-bis}.
Using again \eqref{tis2}, we see that $\widehat{z}^k$ converges a.s. to this same minimizer.
